\RequirePackage{fix-cm}
\documentclass[smallcondensed]{svjour3}
\smartqed

\date{\ }

\usepackage[utf8]{inputenc}
\usepackage{graphicx}
\usepackage{amsmath, amssymb, amsbsy}
\usepackage[colorlinks=true,
allcolors=blue,
linktocpage=true,
breaklinks=true,
bookmarksopen=true,
bookmarksopenlevel=1,
bookmarksnumbered=true,
pdftitle={Notes on {a,b,c}-Modular Matrices},
pdfauthor={Christoph Glanzer, Ingo Stallknecht and Robert Weismantel},
pdfkeywords={Integer optimization, Recognition algorithm, Bounded subdeterminants, Total unimodularity, Bounded minors},
pdflang={en-US}
]{hyperref}
\usepackage{enumerate} 
\usepackage[usenames,dvipsnames]{xcolor}
\usepackage{xspace}
\usepackage{xfrac} 

\newcommand{\Z}{\mathbb{Z}}
\newcommand{\R}{\mathbb{R}}

\newcommand{\inv}{^{-1}}
\newcommand{\tu}{\texttt{TU}\xspace}
\def\T{^\mathrm{T}}
\newcommand{\rank}{\mathrm{rank}}
\newcommand{\ip}{\texttt{IP}\xspace}
\newcommand{\ips}{\texttt{IP}s\xspace}
\newcommand{\lps}{\texttt{LP}s\xspace}
\newcommand{\bip}{\texttt{BIP}\xspace}
\newcommand{\bips}{\texttt{BIP}s\xspace}
\newcommand{\lp}{\texttt{LP}\xspace}
\newcommand{\hnf}{\texttt{HNF}\xspace}

\newcommand{\bfi}[1]{\mathbf{I}_{#1}}
\newcommand{\bfj}[1]{\mathbf{J}_{#1}}

\newcommand{\ineq}[1]{\mathtt{IP}_{\mathrm{\leq}}({#1})\xspace}
\newcommand{\ineqr}[1]{\mathtt{LP}_{\mathrm{\leq}}({#1})\xspace}
\newcommand{\ineqcone}[1]{\mathtt{IP}^{\mathrm{cone}}_{\mathrm{\leq}}({#1})\xspace}

\spnewtheorem{claimenv}{Claim}{}{\itshape}

\makeatletter
\def\@fnsymbol#1{\ifcase#1\or\star\or \dagger\or \ddagger\or
	\mathchar "278\or \mathchar "27B\or \|\or **\or \dagger\dagger
	\or \ddagger\ddagger \else\@ctrerr\fi\relax}
\makeatother

\begin{document}	
	\titlerunning{Notes on $\{a,b,c\}$-Modular Matrices}
	\title{Notes on $\{a,b,c\}$-Modular Matrices\footnote{This is the extended version of a paper published in the Proceedings of the 22nd Conference on Integer Programming and Combinatorial Optimization~\cite{abc_modular_ipco}. It includes all deferred proofs and several examples of $\{a,b,c\}$-modular matrices.}\footnote{This version of the article has been accepted for publication after peer review but is not the Version of Record and does not reflect post-acceptance improvements, or any corrections. The Version of Record is available online at: \url{https://doi.org/10.1007/s10013-021-00520-9}.}}
	\author{Christoph Glanzer\textsuperscript{$\mathsection\ddagger$}\and
		Ingo Stallknecht\textsuperscript{$\mathsection$}\and
		Robert Weismantel\textsuperscript{$\mathsection$}
	}
	\authorrunning{C. Glanzer et al.}
	\institute{	\textsuperscript{$\mathsection$} Department of Mathematics, ETH Zürich, Switzerland.\\\email{\{firstname.lastname\}@ifor.math.ethz.ch}\\~\\\textsuperscript{$\ddagger$} Corresponding author.}
	\maketitle
	
	\begin{abstract}
		Let $A \in \Z^{m \times n}$ be an integral matrix and $a$, $b$, $c \in \Z$ satisfy $a \geq b \geq c \geq 0$. The question is to recognize whether $A$ is \emph{$\{a,b,c\}$-modular}, i.e., whether the set of $n \times n$ subdeterminants of $A$ in absolute value is $\{a,b,c\}$. We will succeed in solving this problem in polynomial time unless $A$ possesses a \emph{duplicative relation}, that is, $A$ has nonzero $n \times n$ subdeterminants $k_1$ and $k_2$ satisfying $2 \cdot |k_1| = |k_2|$. This is an extension of the well-known recognition algorithm for totally unimodular matrices. As a consequence of our analysis, we present a polynomial time algorithm to solve integer programs in standard form over $\{a,b,c\}$-modular constraint matrices for any constants $a$, $b$, and $c$.
		
		\acknowledgement{This work was partially supported by the Einstein Foundation Berlin. We are grateful to Miriam Schlöter for proofreading the manuscript, and to the anonymous reviewers for several helpful comments.}
	\end{abstract}
	
	\newpage
	
	\section{Introduction}

A matrix is called \emph{totally unimodular} (\tu) if all of its subdeterminants are equal to $0$, $1$, or $-1$. Within the past 60 years the community has established a deep and beautiful theory about \tu matrices. A landmark result in the understanding of such matrices is Seymour's decomposition theorem~\cite{Seymour_Decomposition}. It shows that \tu matrices arise from network matrices and two special matrices using row, column, transposition, pivoting, and so-called $k$-sum operations. As a consequence of this theorem it is possible to recognize in polynomial time whether a given matrix is \tu~\cite{Schrijver_IP,truemper1990decomposition}. An implementation of the algorithm in \cite{truemper1990decomposition} by Walter and Truemper~\cite{walter2013implementation} returns a certificate if $A$ is not \tu: For an input matrix with entries in $\{0, \pm 1\}$, the algorithm finds a submatrix $\tilde A$ which is minimal in the sense that $|\det(\tilde A)| = 2$ and every proper submatrix of $\tilde A$ is \tu. We refer to Schrijver~\cite{Schrijver_IP} for a textbook exposition of Seymour's decomposition theorem, a recognition algorithm arising therefrom and further material on \tu matrices.\\

There is a well-established relationship between totally unimodular and \emph{unimodular} matrices, i.e., matrices whose $n \times n$ subdeterminants are equal to $0$, $1$ or $-1$. In analogy to this we define for $A \in \Z^{m \times n}$ and $m \geq n$,
\[ D(A) := \{ |\det(A_{I,\cdot})| \, \colon I \subseteq [m],\, |I| = n\},\]
the set of all $n \times n$ subdeterminants of $A$ in absolute value, where $A_{I,\cdot}$ is the submatrix formed by selecting all rows with indices in $I$. It follows straightforwardly from the recognition algorithm for \tu matrices that one can efficiently decide whether $D(A)\subseteq \{ 1,0 \}$. A technique in~\cite[Section~3]{artmann2017strongly} allows us to recognize in polynomial time whether $D(A) \subseteq \{ 2, 0 \}$. If all $n \times n$ subdeterminants of $A$ are nonzero, the results in~\cite{note2016} can be applied to calculate $D(A)$ given that $\max\{ k\colon k\in D(A) \}$ is constant (cf. Lemma~\ref{lemma:efficient_non_degeneracy_recognition}). Nonetheless, with the exception of these results we are not aware of other instances for which it is known how to determine $D(A)$ in polynomial time.\\

The main motivation for the study of matrices with bounded subdeterminants comes from integer optimization problems (\ips). It is a well-known fact that \ips of the form $\max \{c \T x \colon Ax \leq b,\ x \in \Z^n\}$ for $A \in \Z^{m \times n}$ of full column rank, $b \in \Z^m$ and $c \in \Z^n$ can be solved efficiently if $D(A) \subseteq \{1,0\}$, i.e., if $A$ is unimodular. This naturally leads to the question whether these problems remain efficiently solvable when the assumptions on $D(A)$ are further relaxed. Quite recently in~\cite{artmann2017strongly} it was shown that when $D(A) \subseteq \{2,1,0\}$ and $\rank(A) = n$, integer optimization problems can be solved in strongly polynomial time. Recent results have also led to understand \ips when $A$ is \emph{nondegenerate}, i.e., if $0\not\in D(A)$. The foundation to study the nondegenerate case was laid by Veselov and Chirkov~\cite{veselov2009Intprobimmat}. They describe a polynomial time algorithm to solve \ips if $D(A) \subseteq \{2,1\}$. In~\cite{note2016} the authors showed that \ips over nondegenerate constraint matrices are solvable in polynomial time if the largest $n \times n$ subdeterminant of the constraint matrix is bounded by a constant.

The role of bounded subdeterminants in complexity questions and in the structure of \ips and \lps has also been studied in~\cite{eisenbrand_2017_geometric,glanzer2018number,gribanov2016integer,IntNumb2020}, as well as in the context of combinatorial problems in~\cite{StableSet2020,conforti2020extended,nagele2019submodular}. The sizes of subdeterminants also play an important role when it comes to the investigation of the \emph{diameter of polyhedra}, see~\cite{Dyeretal} and~\cite{Bonifasetal}.

\subsection{Our Results}

A matrix $A \in \Z^{m \times n}$, $m \geq n$, is called \emph{$\{a,b,c\}$-modular} if $D(A) = \{a,b,c\}$, where $a \geq b \geq c \geq 0$.\footnote{For reasons of readability, we will stick to the notation that $a \geq b \geq c$ although the order of these elements is irrelevant.} The paper presents three main results. First, we prove the following structural result for a subclass of $\{a,b,0\}$-modular matrices.
\begin{theorem}[Decomposition Property] \label{thm:format_one_sum}
	Let $a \geq b > 0$, $\gcd(\{a,b\}) = 1$ and assume that $(a,b) \neq (2,1)$. Using row permutations, multiplications of rows by $-1$, and elementary column operations, any $\{a,b,0\}$-modular matrix $A \in \Z^{m \times n}$ can be brought into a block structure of the form
	\begin{align} \label{eq:full_one_sum_layout}
		\begin{bmatrix} L & 0 & \sfrac{0}{a} \\ 0 & R & \sfrac{0}{b} \end{bmatrix},
	\end{align}
	in time polynomial in $n$, $m$, and $\log ||A||_{\infty}$, where $L\in\mathbb{Z}^{m_1 \times n_1}$ and $R\in\mathbb{Z}^{m_2 \times n_2}$ are \tu, $n_1 + n_2 = n-1$, and $m_1 + m_2 = m$. In the representation above the rightmost column has entries in $\{ 0,a \}$ and $\{ 0,b \}$, respectively. The matrix $\begin{bmatrix} L & 0 \\ 0 & R \end{bmatrix}$ contains the $(n-1)$-dimensional unit matrix as a submatrix.
\end{theorem}
The first $n-1$ columns of~\eqref{eq:full_one_sum_layout} are \tu since they form a $1$-sum of two \tu matrices (see \cite[Chapter~19.4]{Schrijver_IP}). This structural property lies at the core of the following recognition algorithm. We say that a matrix $A$ possesses a \emph{duplicative relation} if it has nonzero $n \times n$ subdeterminants $k_1$ and $k_2$ satisfying $2 \cdot |k_1| = |k_2|$.
\begin{theorem}[Recognition Algorithm] \label{thm:abc_recognition}
	There exists an algorithm that solves the following recognition problem in time polynomial in $n$, $m$, and $\log ||A||_{\infty}$: Either, calculate $D(A)$, or give a certificate that $|D(A)| \geq 4$, or return a duplicative relation.
\end{theorem}
For instance, Theorem~\ref{thm:abc_recognition} cannot be applied to check whether a matrix is $\{4,2,0\}$-modular, but it can be applied to check whether a matrix is $\{3,1,0\}$- or $\{6,4,0\}$-modular. More specifically, Theorem~\ref{thm:abc_recognition} recognizes $\{a,b,c\}$-modular matrices unless $(a,b,c) = (2 \cdot k,k,0)$, $k \in \Z_{\geq 1}$. In particular, this paper does not give a contribution as to whether so-called \emph{bimodular matrices} (the case $k=1$) can be recognized efficiently. This is because Theorem~\ref{thm:format_one_sum} excludes the case $(a,b) = (2,1)$.

The decomposition property established in Theorem~\ref{thm:format_one_sum} is a major ingredient for the following optimization algorithm to solve standard form \ips over $\{a,b,c\}$-modular constraint matrices for any constant $a \geq b \geq c \geq 0$.
\begin{theorem}[Optimization Algorithm]\label{thm:standard_optimize}
	Consider a standard form integer program of the form
	\begin{align}
		\label{eq:standard_form_ip_statement} \max \{c\T x \colon Bx = b,\ x \in \Z^n_{\geq 0} \},
	\end{align}
	for $b \in \Z^m$, $c \in \Z^n$, and $B \in \Z^{m \times n}$ of full row rank, where $D(B\T)$ is constant, i.e., $\max\{ k\colon k\in D(B\T) \}$ is constant.\footnote{Note that we use $D(B\T)$ instead of $D(B)$ since $B$ has full row rank and we refer to its $m \times m$ subdeterminants.} Then, in time polynomial in $n$, $m$, and the encoding size of the input data, one can solve~\eqref{eq:standard_form_ip_statement} or output that $|D(B\T)| \geq 4$.
\end{theorem}
Notably, in Theorem~\ref{thm:standard_optimize}, the assumption that $D(B\T)$ is constant can be dropped if $B$ is \emph{degenerate}, i.e., if $0 \in D(B\T)$.

\subsection{Three Examples of $\{a,b,c\}$-Modular Matrices}
\begin{itemize}
	\item \textit{Generalized network flow.} Let $G = (V,E)$ be a directed graph whose vertices can be partitioned as $V = S \cup \{v\} \cup T$ such that no arc runs between $S$ and $T$, and no arc runs from $v$ to $S$. Consider a generalized network flow problem in $G$, where $s \in S$ and $t \in T$, with capacities $u \colon E \rightarrow \Z_{> 0}$ and gains
	\[
	\gamma(e) := \begin{cases}
		\ \frac{a}{b} & \mbox{if } e \ \text{runs from} \ S \ \text{to} \ v,\\
		\ 1 & \mbox{otherwise},
	\end{cases}
	\]
	where $a$, $b \in \Z_{> 0}$. Then, the natural formulation of finding a maximal $s$-$t$ flow in $G$ with respect to $u$ and $\gamma$ is an integer optimization problem whose constraint matrix $A$ satisfies $D(A\T) \subseteq \{a,b,0\}$ after multiplying the constraint corresponding to $v$ by $b$.\\
	
	\item \textit{Perfect $d$-matching.} Let $G = (V,E)$ be an undirected graph with edge capacities $u \colon E \rightarrow \Z_{> 0} \cup \{\infty\}$, weights $c \colon E \rightarrow \R$, and numbers $d \colon V \rightarrow \Z_{> 0}$. Then, the \emph{perfect $d$-matching problem} is the problem of finding a function $f \colon E \rightarrow \Z_{\geq 0}$ of maximal cost $\sum_{e \in E} f(e) \cdot c(e)$ which satisfies $f(e) \leq u(e)$ for all $e \in E$ and $\sum_{e \in \delta(v)} f(e) = d(v)$ for all $v \in V$, see \cite{hupp2017integer,kortevygen}.\\
	
	Consider the following variation of this problem: Let $G_1 = (V_1, E_1)$ and $G_2 = (V_2, E_2)$ be two bipartite graphs with edge capacities $u \colon E_1 \cup E_2 \rightarrow \Z_{> 0} \cup \{\infty\}$, weights $c \colon E_1 \cup E_2 \rightarrow \R$, and numbers $d \colon V_1 \cup V_2 \rightarrow \Z_{> 0}$. Let $a$, $b \in \Z_{> 0}$, and $g \in \Z$. Fix edges $e_1 \in E_1$ and $e_2 \in E_2$. Then, the natural formulation of finding an optimal, perfect $d$-matching $f$ in $(V_1 \cup V_2, E_1 \cup E_2)$ w.r.t. $u$, $c$, and $d$ and under the additional constraint that
	\[ \pm a \cdot f(e_1) \pm b \cdot f(e_2) = g, \]
	is an integer optimization problem whose constraint matrix $A$ satisfies $D(A\T) \subseteq \{a,b,0\}$. The signs of $a$ and $b$ can be different.\\
	
	\item \textit{Edge-weighted vertex cover.} Given an undirected graph $G = (V,E)$ and weights $w \colon E \rightarrow \Z_{\geq 0}$, an \emph{edge-weighted vertex cover} is a function $f \colon V \rightarrow \Z_{\geq 0}$ such that for each $e = \{u,v\} \in E$, $f(u) + f(v) \geq w(e)$.\\
	
	Consider the following variation of this problem: Let $G_1 = (V_1, E_1)$ and $G_2 = (V_2, E_2)$ be two bipartite graphs with weights $w \colon E_1 \cup E_2 \rightarrow \Z_{\geq 0}$ and costs $c \colon V_1 \cup V_2 \rightarrow \R_{\geq 0}$. Let $a$, $b \in \Z_{> 0}$. Then, the natural formulation of finding a minimal edge-weighted vertex cover $f$ in $G_1 \cup G_2$ w.r.t. $c$, where the constraints are replaced by 
	\begin{align*}
		f(u) + f(v) \pm a \cdot z \geq & ~ w(e), \ \forall e = \{u,v\} \in E_1,\\
		f(u) + f(v) \pm b \cdot z \geq & ~ w(e), \ \forall e = \{u,v\} \in E_2,
	\end{align*}
	is an integer optimization problem whose constraint matrix is $\{a,b\}$- or $\{a,b,0\}$-modular.\footnote{The proof of Theorem~\ref{thm:standard_optimize} contains an algorithm which solves such \ips in inequality form efficiently if the condition $\gcd(a,b) = 1$ is fulfilled.} Note in particular that the signs of the two constraints can be different.
\end{itemize}
	
	\section{Notation and Preliminaries}
	\label{sct:preliminaries}

For $k \in \Z_{\geq 1}$, $[k] := \{1, \ldots, k\}$. $\mathcal I_n$ is the $n$-dimensional unit matrix, where we leave out the subscript if the dimension is clear from the context. For a matrix $A \in \Z^{m \times n}$, we denote by $A_{i,\cdot}$ the $i$-th row of $A$. For a subset $I$ of $[m]$, $A_{I,\cdot}$ is the submatrix formed by selecting all rows with indices in $I$, in increasing order. An analogous notation is used for the columns of $A$. For $k \in \Z$, we write $I_k :=\{ i\in [m]\colon A_{i,n} = k \}$, the indices of the rows whose $n$-th entry is equal to $k$. Set $||A||_{\infty}:=\max_{i \in [m], j \in [n]} |A_{i,j}|$. For simplicity, we assume throughout the document that the input matrix $A$ to any recognition algorithm satisfies $m \geq n$ and $\rank(A) = n$ as $\rank(A) < n$ implies that $D(A) =\{ 0 \}$.
Left-out entries in figures and illustrations are equal to zero.

$D(A)$ is preserved under elementary column operations, permutations of rows and multiplications of rows by $-1$. By a series of elementary column operations on $A$, any nonsingular $n \times n$ submatrix $B$ of $A$ can be transformed to its Hermite normal form (\hnf), in which $B$ becomes a lower-triangular, nonnegative submatrix with the property that each of its rows has a unique maximum entry on its main diagonal~\cite{Schrijver_IP}. This can be done in time polynomial in $n$, $m$, and $\log ||A||_{\infty}$~\cite{storjohann1996asymptotically}.

At various occasions we will make use of the following simple adaptation of the \hnf described in~\cite[Section~3]{note2016}.
Begin by choosing a nonsingular $n \times n$ submatrix $B$ of $A$: Either, the choice of $B$ will be clear from the context or otherwise, choose $B$ to be any nonsingular $n \times n$ submatrix of $A$. Permute the rows of $A$ such that $A_{[n],\cdot} = B$. Apply elementary column operations to $A$ such that $B$ is in \hnf. After additional row and column permutations and multiplications of rows by $-1$,
\begin{align}\label{eq:HNF}
A = \left[ \begin{matrix}
1 &&&&&& ~ \\
& \ddots &&&&& ~ \\
&&1&&&& ~ \\
*&\cdots&*& \delta_1 &&& ~ \\
\vdots&\vdots&\vdots&\ddots& \ddots && ~ \\
*&\cdots&\cdots&\cdots&*& \delta_l& ~ \\
*&\cdots&\cdots&\cdots&\cdots&  * & A_{n,n} \\
A_{n+1,1} &\cdots&\cdots&\cdots&\cdots& A_{n+1,n-1} & A_{n+1,n} \\
\vdots & \vdots & \vdots & \vdots& \vdots& \vdots& \vdots \\
A_{m,1} &\cdots&\cdots&\cdots&\cdots& A_{m,n-1} & A_{m,n}
\end{matrix} \right],
\end{align}
where $A_{\cdot,n} \geq 0$, $|\det(B)| = (\prod_{i=1}^l \delta_i) \cdot A_{n,n}$, all entries marked by $\ast$ are numbers between $0$ and the corresponding entry on the main diagonal minus one, and $\delta_i \geq 2$ for all $i \in [l]$. In particular, the rows of $A_{[n],\cdot}$ whose entries on the main diagonal are strictly larger than one are at positions $n-l, \ldots, n$.\\

We note that it is not difficult to efficiently recognize nondegenerate matrices given a constant upper bound on $|D(A)|$. This will allow us to exclude the nondegenerate case in all subsequent algorithms. We wish to emphasize that the results in~\cite{note2016} can be applied to solve this task given that $\max\{ k\colon k\in D(A) \}$ is constant.
\begin{lemma} \label{lemma:efficient_non_degeneracy_recognition}
	Given a constant $d\in\mathbb{Z}$, there exists an algorithm that solves the following recognition problem in time polynomial in $n$, $m$, and $\log ||A||_{\infty}$: Either, calculate $D(A)$, or give a certificate that $|D(A)| \geq d+1$, or that $0 \in D(A)$.
\end{lemma}
\begin{proof}
	We first prove that for $n \geq 2$, any nondegenerate matrix $A \in \Z^{m \times n}$ with $|D(A)| \leq d$ has at most $(n-1)+d \cdot (2\cdot d+1)$ rows. Given such a matrix $A$, transform it to~\eqref{eq:HNF}. We introduce the following family of $2 \times 2$ subdeterminants: For any $i$, $j \geq n-1$, $i \neq j$, set
	\[
		\theta_{i,j} := \det \begin{bmatrix} A_{i,n-1} & A_{i,n} \\ A_{j,n-1} & A_{j,n} \end{bmatrix}.
	\]
	Any such subdeterminant can be extended to an $n \times n$ subdeterminant $\Theta_{i,j}$ of $A$ by appending rows $1, \ldots, n-2$ and columns $1, \ldots, n-2$ to $\theta_{i,j}$. Applying Laplace expansion yields $\Theta_{i,j} = (\prod_{k=1}^{l-1} \delta_k) \cdot \theta_{i,j}$.
	
	Starting with row $n$, partition the rows of $A$ into submatrices (bins) $A[1]$, \ldots, $A[s]$ such that the rows with the same entry in the $n$-th column are in the same bin. Thus, $s$ is the number of different entries in the $n$-th column, starting from the $n$-th row. Denote by $I[i]$ the indices of the rows in $A[i]$, i.e., $A_{I[i],\cdot} = A[i]$. In what follows, we first derive a bound on $s$ and subsequently bound the number of rows in each bin.
	
	Regarding the former, as $A_{n-1,n} = 0$, it holds that $\theta_{n-1,j} = \delta_l \cdot A_{j,n}$ for any $j \geq n$. This implies that $A_{j,n} \neq 0$ for any $j \geq n$ as otherwise, $\Theta_{n-1,j} = 0$, contradicting the nondegeneracy of $A$. As we have assumed that $A_{\cdot,n} \geq 0$, it follows that $A_{\cdot,n} > 0$. Therefore, $\theta_{n-1,j} > 0$ and $\Theta_{n-1,j} > 0$ for any $j \geq n$. Varying $j$ among different bins induces pairwise different and nonnegative subdeterminants $\Theta_{n-1,j}$, i.e., $s \leq d$.
	
	Next, we derive a bound on the number of rows of each bin. Choose an arbitrary bin $k \in [s]$. Let $i$ be the smallest element of $I[k]$ and let $j_1$, $j_2 \in I[k] \setminus \{i\}$, $j_1 \neq j_2$. As $A_{i,n} = A_{j_1,n} = A_{j_2,n}$, it holds that
	\begin{align*}
		\text{(i)} & ~~ \theta_{j_1,j_2} = A_{i,n} \cdot (A_{j_1,n-1} - A_{j_2,n-1}),\\
		\text{(ii)} & ~~ \theta_{i,j_1} = A_{i,n} \cdot (A_{i,n-1} - A_{j_1,n-1}),\\
		\text{(iii)} & ~~ \theta_{i,j_2} = A_{i,n} \cdot (A_{i,n-1} - A_{j_2,n-1}).
	\end{align*}
	From (i), it follows that $A_{j_1,n-1} \neq A_{j_2,n-1}$ as otherwise, $\theta_{j_1,j_2} = 0$. Regarding (ii) and (iii), note that $A_{j_1,n-1} \neq A_{j_2,n-1} \Leftrightarrow \theta_{i,j_1} \neq \theta_{i,j_2}$. This is equivalent to $\Theta_{i,j_1} \neq \Theta_{i,j_2}$. In other words, every element of $I[k] \setminus \{i\}$ yields a different $n \times n$ subdeterminant. Since those $n \times n$ subdeterminants are not necessarily different in absolute value, $|I[k] \setminus \{i\}| \leq 2 \cdot d$ and $|I[k]| \leq 2 \cdot d+1$. Combining this bound with our bound on the number of bins, we obtain $m\leq (n-1) + d \cdot (2 \cdot d+1)$.\\
	
	From the result above we straightforwardly deduce the following simple recognition algorithm: Let $A$ be the input matrix. If $n=1$ perform an exhaustive search. If $m \leq (n-1)+d \cdot (2 \cdot d+1)$, enumerate all $\binom{(n-1)+d \cdot (2 \cdot d+1)}{n}$ $n \times n$ subdeterminants. Otherwise, we find at least $d+1$ elements of $D(A)$ or an $n \times n$ subdeterminant equal to zero as described above.\qed
\end{proof}
	
	\section{Proof of Theorem \ref{thm:format_one_sum}}
	\label{sct:struct_prop}

Transform $A$ to~\eqref{eq:HNF}, thus $A_{\cdot,n} \geq 0$. As a first step we show that we may assume that $A_{n,n} > 1$ without loss of generality. For this purpose, assume that $A_{n,n} = 1$, i.e., that $l=0$. This implies that $A_{[n],\cdot} = \mathcal I_n$. Consequently, any nonsingular submatrix $B$ of $A$ can be extended to an $n \times n$ submatrix with the same determinant in absolute value by (a) appending unit vectors from the topmost $n$ rows of $A$ to $B$ and (b) for each unit vector appended in step (a), by appending the column to $B$ in which this unit vector has its nonzero entry. By Laplace expansion, the $n \times n$ submatrix we obtain admits the same determinant in absolute value. Therefore, if we identify any submatrix of $A$ with determinant larger than one in absolute value, we can transform $A$ once more to~\eqref{eq:HNF} with respect to its corresponding $n \times n$ submatrix which yields $A_{n,n} > 1$ as desired. To find a subdeterminant of $A$ of absolute value larger than one, if present, test $A$ for total unimodularity. If the test fails, it returns a desired subdeterminant. If the test returns that $A$ is \tu, then $a=b=1$ and $A$ is already of the form $\begin{bmatrix} L & \sfrac{0}{1} \end{bmatrix}$ for $L$ \tu.

The $n$-th column is not divisible by any integer larger than one as otherwise, all $n \times n$ subdeterminants of $A$ would be divisible by this integer, contradicting $\gcd(\{a,b\}) = 1$. In particular, since $A_{n,n} > 1$, $A_{\cdot,n}$ is not divisible by $A_{n,n}$. This implies that there exists an entry $A_{k,n} \neq 0$ such that $A_{k,n} \neq A_{n,n}$. Thus, there exist two $n \times n$ subdeterminants, $\det(A_{[n],\cdot})$ and $\det(A_{[n-1] \cup k,\cdot})$, of different absolute value. This allows us to draw two conclusions: First, the precondition $A_{n,n} > 1$ which we have established in the previous paragraph implies $a > b$. We may therefore assume for the rest of the proof that $a \geq 3$ as $\{a,b\} \neq \{2,1\}$. Secondly, it follows that $l=0$: For the purpose of contradiction assume that $l \geq 1$. Then the aforementioned two subdeterminants $\det(A_{[n],\cdot})$ and $\det(A_{[n-1] \cup k,\cdot})$ are both divisible by $\prod_{i=1}^l \delta_i > 1$. Since one of those subdeterminants must be equal to $\pm a$ and the other must be equal to $\pm b$, this contradicts our assumption of $\gcd(\{a,b\}) = 1$.

As a consequence of $l=0$, the topmost $n-1$ rows of $A$ form unit vectors. Thus, any subdeterminant of rows $n, \ldots, m$ of $A$ which includes elements of the $n$-th column can be extended to an $n \times n$ subdeterminant of the same absolute value by appending an appropriate subset of these unit vectors. To further analyze the structure of $A$ we will study the $2 \times 2$ subdeterminants
\begin{align} \label{eq:2x2_extended_blockmat}
	\theta^h_{i,j} := \det \begin{bmatrix} A_{i,h} & A_{i,n} \\ A_{j,h} & A_{j,n} \end{bmatrix},
\end{align}
where $i$, $j \geq n$, $i \neq j$, and $h \leq n-1$. It holds that $|\theta^h_{i,j}| \in \{a,b,0\}$.
\begin{claimenv} \label{claim:claim1}
	There exists a sequence of elementary column operations and row permutations which turn $A$ into the form
	\begin{align} \label{eq:partial_one_sum_layout}
	\begin{bmatrix}
	\mathcal I_{n-1} & \ \\
	\sfrac{0}{\pm \hspace{-0.25em} 1} & a \\
	\sfrac{0}{\pm \hspace{-0.25em} 1} & b \\
	\sfrac{0}{\pm \hspace{-0.25em} 1} & 0 \end{bmatrix},
	\end{align}
	where $[\begin{smallmatrix}\sfrac{0}{\pm \hspace{-0.25em} 1} & a\end{smallmatrix}]$, $[\begin{smallmatrix}\sfrac{0}{\pm \hspace{-0.25em} 1} & b\end{smallmatrix}]$, and $[\begin{smallmatrix}\sfrac{0}{\pm \hspace{-0.25em} 1} & 0\end{smallmatrix}]$ are submatrices consisting of rows whose first $n-1$ entries lie in $\{0, \pm 1\}$, and whose $n$-th entry is equal to $a$, $b$, or $0$, respectively.
\end{claimenv}
\begin{proof}[Proof of Claim~\ref{claim:claim1}]
	We start by analyzing the $n$-th column of $A$. To this end, note that $|\det(A_{[n-1] \cup k,\cdot})| = |A_{k,n}| \in \{a,b,0\}$ for $k \geq n$. It follows that $A_{\cdot,n} \in \{a,b,0\}^m$ as $A_{\cdot,n} \geq 0$. In addition, by what we have observed two paragraphs earlier, at least one entry of $A_{\cdot,n}$ must be equal to $a$ and at least one entry must be equal to $b$. Sort the rows of $A$ by their respective entry in the $n$-th column as in~\eqref{eq:partial_one_sum_layout}.
	
	In the remaining proof we show that after column operations, $A_{\cdot,[n-1]} \in \{0, \pm 1\}^{m \times (n-1)}$. To this end, let $h \in [n-1]$ be an arbitrary column index. Recall that $I_k :=\{ i\in [m]\colon A_{i,n} = k \}$. We begin by noting a few properties which will be used to prove the claim. For $i \in I_a$ and $j \in I_b$, it holds that $\theta^h_{i,j} = b \cdot A_{i,h} - a \cdot A_{j,h} \in \{\pm a, \pm b, 0\}$. These are Diophantine equations which are solved by
	\begin{align} \label{eq:dioph_structural}
		\text{(i)} & ~~ (A_{i,h},A_{j,h}) = (k \cdot a,\mp1+k \cdot b),\ k \in \Z,\nonumber\\
		\text{(ii)} & ~~ (A_{i,h},A_{j,h}) = (\pm 1 + k \cdot a,k \cdot b),\ k \in \Z,\\
		\text{(iii)} & ~~ (A_{i,h},A_{j,h}) = (k \cdot a,k \cdot b),\ k \in \Z.\nonumber
	\end{align}
	Furthermore, for any $i_1$, $i_2 \in I_a$, it holds that $\theta^h_{i_1,i_2} = a \cdot (A_{i_1,h} - A_{i_2,h})$. Since this quantity is a multiple of $a$ and since $\theta^h_{i_1 ,i_2}\in \{\pm a, \pm b, 0\}$, it follows that
	\begin{align} \label{eq:dist1_structural}
		|A_{i_1,h} - A_{i_2,h}| \leq 1,~\forall i_1, i_2 \in I_a.
	\end{align}
	We now perform a column operation on $A_{\cdot,h}$: Let us fix arbitrary indices $p \in I_a$ and $q \in I_b$. The pair $(A_{p,h},A_{q,h})$ solves one of the three Diophantine equations for a fixed $k$. 
	Add $(-k) \cdot A_{\cdot,n}$ to $A_{\cdot,h}$. Now, $(A_{p,h},A_{q,h}) \in \{ (0,\mp 1), (\pm 1,0), (0,0) \}$. We claim that as a consequence of this column operation, $A_{\cdot,h} \in \{0, \pm 1\}^m$.
	
	We begin by showing that $A_{I_a,h}$ and $A_{I_b,h}$ have entries in $\{0, \pm 1\}$. First, assume for the purpose of contradiction that there exists $j \in I_b$ such that $|A_{j,h}| > 1$. This implies that the pair $(A_{p,h},A_{j,h})$ satisfies~\eqref{eq:dioph_structural} for $|k| \geq 1$. As $a \geq 3$, this contradicts that $A_{p,h} \in \{0, \pm 1\}$. Secondly, assume for the purpose of contradiction that there is $i \in I_a$ such that $|A_{i,h}| > 1$.\footnote{We cannot use the same approach as for the first case as $b$ could be equal to $1$ or $2$.} As $|A_{p,h}| \leq 1$, it follows from~\eqref{eq:dist1_structural} that $|A_{i,h}| = 2$ and that $|A_{p,h}| = 1$. Therefore, as either $A_{p,h}$ or $A_{q,h}$ must be equal to zero, $A_{q,h} = 0$. This implies that $|\theta^h_{i,q}| = 2 \cdot b$ which is a contradiction to $\theta^h_{i ,q}\in \{\pm a, \pm b, 0\}$ as $A$ has no duplicative relation, i.e., $2 \cdot b \neq a$.
	
	It remains to prove that $A_{I_0,h}$ has entries in $\{0, \pm 1\}$. For the purpose of contradiction, assume that there exists $i\in I_0$, $i \geq n$, such that $|A_{i,h}| \geq 2$. Choose any $s \in I_a$. Then,
	\[ \theta^h_{s,i} = \det \begin{bmatrix} A_{s,h} & a \\ A_{i,h} & 0 \end{bmatrix} = -a \cdot A_{i,h}, \]
	which is larger than $a$ in absolute value, a contradiction.\qed
\end{proof}
In the next claim, we establish the desired block structure~\eqref{eq:full_one_sum_layout}. We will show afterwards that the blocks $L$ and $R$ are \tu.
\begin{claimenv} \label{claim:claim2}
	There exists a sequence of row and column permutations which turn $A$ into~\eqref{eq:full_one_sum_layout} for matrices $L$ and $R$ with entries in $\{0, \pm 1\}$.
\end{claimenv}
\begin{proof}[Proof of Claim~\ref{claim:claim2}]
	For reasons of simplicity, we assume that $A$ has no rows of the form $\begin{bmatrix} 0 & \cdots & 0 & a \end{bmatrix}$ or $\begin{bmatrix} 0 & \cdots & 0 & b \end{bmatrix}$ as such rows can be appended to $A$ while preserving the properties stated in this claim. We construct an auxiliary graph $G = (V,E)$, where we introduce a vertex for each nonzero entry of $A_{\cdot,[n-1]}$ and connect two vertices if they share the same row or column index. Formally, set
	\begin{align*}
		V := & \ \{ (i,j) \in [m] \times [n-1] \colon A_{i,j} \neq 0 \},\\
		E := & \ \{ \{(i_1, j_1),(i_2 ,j_2)\} \in V \times V \colon i_1 = i_2\ \text{or} \ j_1 = j_2,\ (i_1,j_1) \neq (i_2,j_2) \}.
	\end{align*}
	Let $K_1, \ldots, K_k \subseteq V$ be the vertex sets corresponding to the connected components in $G$. For each $l \in [k]$, set
	\begin{align*}
		\bfi l := \{ i \in [m] \colon \exists j \in [n-1] \ \text{s.t.} \ (i,j) \in K_l \},\\
		\bfj l := \{ j \in [n-1] \colon \exists i \in [m] \ \text{s.t.} \ (i,j) \in K_l \}.
	\end{align*} 
	These index sets form a partition of $[m]$, resp. $[n-1]$: Since every row of $A_{\cdot,[n-1]}$ is nonzero, $\bigcup_{l \in [k]} \bfi l = [m]$ and since $\rank(A) = n$, every column of $A$ contains a nonzero entry, i.e., $\bigcup_{l \in [k]} \bfj l = [n-1]$. Furthermore, by construction, it holds that $\bfi p \cap \bfi q = \emptyset$ and $\bfj p \cap \bfj q = \emptyset$ for all $p \neq q$. The entries $A_{i,j}$ for which $(i,j) \notin \bigcup_{l\in [k]} (\bfi l \times \bfj l)$ are equal to zero. Therefore, sorting the rows and columns of $A_{\cdot,[n-1]}$ with respect to the partition formed by $\bfi l$, resp. $\bfj l$, $l \in [k]$, yields
	\begin{equation} \label{eq:block_structure}
		A_{\cdot,[n-1]} = \begin{bmatrix} A_{\bfi 1, \bfj 1} & 0 & \cdots & 0\\ 0 & \ddots & \ddots & \vdots \\ \vdots & \ddots & \ddots & 0 \\ 0 & \cdots & 0 & A_{\bfi k, \bfj k} \end{bmatrix}.
	\end{equation} 
	In what follows we show that for all $l \in [k]$, it holds that either $\bfi l \cap I_a = \emptyset$ or that $\bfi l \cap I_b = \emptyset$. From this property the claim readily follows: We obtain the form~\eqref{eq:full_one_sum_layout} from~\eqref{eq:block_structure} by permuting the rows and columns such that the blocks which come first are those which correspond to the connected components $K_l$ with $\bfi l \cap I_a \neq \emptyset$. For the purpose of contradiction, assume that there exists $l \in [k]$ such that $\bfi l \cap I_a \neq \emptyset$ and $\bfi l \cap I_b \neq \emptyset$. Set $K_l^a := \{(i,j) \in K_l \colon i \in I_a\}$ and $K_l^b := \{(i,j) \in K_l \colon i \in I_b\}$. Among all paths in the connected component induced by $K_l$ which connect the sets $K_l^a$ and $K_l^b$, let $P := \{(i^{(1)},j^{(1)}), \ldots, (i^{(t)},j^{(t)})\}$ be a shortest one. By construction, $(i^{(1)},j^{(1)})$ is the only vertex of $P$ which lies in $K_l^a$ and $(i^{(t)},j^{(t)})$ is the only vertex of $P$ which lies in $K_l^b$. This implies that $P$ starts and ends with a change in the first component, i.e., $i^{(1)} \neq i^{(2)}$ and $i^{(t-1)} \neq i^{(t)}$. Furthermore, since $P$ has minimal length, it follows from the construction of the edges that it alternates between changing first and second components, i.e., for all $s=1, \ldots, t-2$, $i^{(s)} \neq i^{(s+1)} \Leftrightarrow j^{(s+1)} \neq j^{(s+2)}$ and $j^{(s)} \neq j^{(s+1)} \Leftrightarrow i^{(s+1)} \neq i^{(s+2)}$.
	
	Define $B := A_{\{i^{(1)}, \ldots, i^{(t)}\}, \{j^{(1)}, \ldots, j^{(t)} ,n\}}$. $B$ is a square matrix with $\frac{t+2}{2}$ rows and columns for the following reason: $P$ starts with a change of rows, after which the remaining path of length $t-2$ starts with a change of columns, ends with a change of rows and alternates between row and column changes in between. Thus, $t-2$ is even and in total $P$ changes rows $1 + \frac{t-2}{2}$ times and columns $\frac{t-2}{2}$ times. It follows that $B$ has $1 + \frac{t-2}{2} + 1$ rows and $\frac{t-2}{2} + 1$ columns excluding the column with entries of $A_{\cdot,n}$. Permute the rows and columns of $B$ such that they are ordered with respect to the order $i^{(1)}, \ldots, i^{(t)}$ and $j^{(1)}, \ldots, j^{(t)},n$. We claim that $B$ is of the form
	\[ B = \begin{bmatrix} \pm 1 & \ & \ & \ & \ & a \\
	\pm 1 & \pm 1 & \ & \ & \ & 0 \\
	\ & \pm 1 & \ddots & \ & \ & \vdots \\
	\ & \ & \ddots & \pm 1 & \ & \vdots \\
	\ & \ & \ & \pm 1 & \pm 1 & 0 \\
	\ & \ & \ & \ & \pm 1 & b \end{bmatrix}.\]
	To see this, first observe that the entries on the main diagonal and the diagonal below are equal to $\pm 1$ and that the entries $B_{1,\frac{t+2}{2}}$ and $B_{\frac{t+2}{2},\frac{t+2}{2}}$ are equal to $a$ or $b$, respectively, by construction. As we have observed above, $(i^{(1)},j^{(1)})$ is the only vertex of $P$ which lies in $K_l^a$ and $(i^{(t)},j^{(t)})$ is the only vertex of $P$ which lies in $K_l^b$, implying that all other entries of the rightmost column, $B_{2,\frac{t+2}{2}}, \ldots, B_{\frac{t+2}{2} - 1,\frac{t+2}{2}}$, are equal to zero. For the purpose of contradiction, assume that any other entry of $B$, say $B_{v,w}$, is nonzero. Let us consider the case that $w > v$, i.e., that $B_{v,w}$ lies above the main diagonal of $B$. Among all vertices of $P$ whose corresponding entries lie in the same row as $B_{v,w}$, let $(i^{(s)},j^{(s)})$ be the vertex with minimal $s$. Among all vertices of $P$ whose corresponding entries lie in the same column as $B_{v,w}$, let $(i^{(s')},j^{(s')})$ be the vertex with maximal $s'$. Then, consider the following path: Connect $(i^{(1)},j^{(1)})$ with $(i^{(s)},j^{(s)})$ along $P$; then take the edge to the vertex corresponding to $B_{v,w}$; take the edge to $(i^{(s')},j^{(s')})$; then connect $(i^{(s')},j^{(s')})$ with $(i^{(t)},j^{(t)})$ along $P$. This path is shorter than $P$, a contradiction. An analogous argument yields a contradiction if $w < v-1$, i.e., if $B_{v,w}$ lies in the lower-triangular part of $B$. Thus, $B$ is of the form stated above. By Laplace expansion applied to the last column of $B$, $\det(B) = \pm a \pm b \notin \{\pm a, \pm b, 0\}$ as $2 \cdot b \neq a$. Since $B$ can be extended to an $n \times n$ submatrix of $A$ of the same determinant in absolute value, this is a contradiction.
	
	Regarding the computational running time of this operation, note that $G$ can be created and its connected components can be calculated in time polynomial in $n$ and $m$. A subsequent permutation of the rows and columns as noted above yields the desired form.\qed
\end{proof}
It remains to show that $L$ and $R$ are \tu. We will prove the former, the proof of the latter property is analogous. Assume for the purpose of contradiction that $L$ is not \tu. As the entries of $L$ are all $\pm 1$ or $0$, it contains a submatrix $\tilde A$ of determinant $\pm 2$~\cite[Theorem~19.3]{Schrijver_IP}. Extend $\tilde A$ by appending the corresponding entries of $A_{\cdot,n}$ and by appending an arbitrary row of $A_{I_b,\cdot}$. We obtain a submatrix of the form
\[ \begin{bmatrix} \tilde A & \sfrac{0}{a} \\ 0 & b \end{bmatrix}, \]
whose determinant in absolute value is $2 \cdot b$, a contradiction as $A$ has no duplicative relations. Similarly, one obtains a submatrix of determinant $\pm 2 \cdot a$ if $R$ is not \tu.\qed

	\section{Proof of Theorem \ref{thm:abc_recognition}}
	\label{sct:main_proof}

In this section, the following recognition problem will be of central importance: For numbers $a \geq b \geq c \geq 0$, we say that an algorithm \emph{tests for $\{a,b,c\}$-modularity} if, given an input matrix $A \in \Z^{m \times n}$ ($m \geq n$), it checks whether $D(A) = \{a,b,c\}$. As a first step we state the following lemma which allows us to reduce the $\gcd$ in all subsequent recognition algorithms. The proof uses a similar technique as was used in~\cite[Remark~5]{gribanov2020parametric}.
\begin{lemma}[{cf.~\cite[Remark~5]{gribanov2020parametric}}] \label{lemma:wlog_gcd_1}
	Let $a \geq b > 0$ and $\gamma := \gcd(\{a,b\})$. There exists an algorithm with running time polynomial in $n$, $m$, and $\log ||A||_{\infty}$ which reduces testing for $\{a,b,0\}$-modularity to testing for $\{\frac a \gamma ,\frac b \gamma,0\}$-modularity or returns that $\gcd(D(A)) \neq \gcd(\{a,b\})$, where $A$ is the input matrix.
\end{lemma}
\begin{proof}
	Calculate the transformation matrices which transform $A$ to its Smith normal form: Find $P \in \Z^{m \times m}$ and $Q \in \Z^{n \times n}$ unimodular such that
	\[ PAQ = \begin{bmatrix} S \\ 0 \end{bmatrix}, \]
	where $S \in \Z^{n \times n}$ is a diagonal matrix satisfying $\prod_{i=1}^n S_{i,i} = \gcd(D(A))$, cf.~\cite{Schrijver_IP}. The matrices $P$ and $Q$ can be calculated in time polynomial in $n$, $m$ and $\log ||A||_{\infty}$, see~\cite{Storjohann_PhD}. It must hold that $\prod_{i=1}^n S_{i,i} = \gamma$, otherwise $\gcd(D(A)) \neq \gcd(\{a,b\})$, i.e., we have found a certificate that $A$ is not $\{a,b,0\}$-modular. Since $S$ and $P$ are invertible, we obtain that
	\[ AQS \inv = (P \inv)_{\cdot,[n]}.\]
	As $P$ is unimodular, $P \inv$ is integral, i.e., $AQS \inv$ is integral. Since $Q$ is unimodular, $AQ$ corresponds to performing elementary column operations on $A$, i.e., $D(AQ) = D(A)$. As $AQS \inv$ is integral, the $i$-th column of $AQ$ is divisible by $S_{i,i}$, $1 \leq i \leq n$. Multiplying $A$ by $S\inv$ from the right corresponds to performing these divisions. We conclude that $D(AQS\inv) = \frac{1}{\gcd(D(A))} \cdot D(A)$. Testing $A$ for $\{a,b,0\}$-modularity is therefore equivalent to testing $AQS \inv$ for $\{\frac a \gamma, \frac b \gamma, 0\}$-modularity.\qed
\end{proof}
We proceed by using the decomposition established in Theorem~\ref{thm:format_one_sum} to construct an algorithm which tests for $\{a,b,0\}$-modularity if $a \geq b > 0$ and if $2 \cdot b \neq a$, i.e., if there are no duplicative relations. We will see that this algorithm quickly yields a proof of Theorem~\ref{thm:abc_recognition}. The main difference between the following algorithm and Theorem~\ref{thm:abc_recognition} is that in the former, $a$ and $b$ are fixed input values while in Theorem~\ref{thm:abc_recognition}, $a$, $b$ (and $c$) also have to be determined.
\begin{lemma} \label{lemma:efficient_0ab_modularity_recognition}
	Let $a \geq b > 0$ such that $2 \cdot b \neq a$. There exists an algorithm with running time polynomial in $n$, $m$, and $\log ||A||_{\infty}$ which tests for $\{a,b,0\}$-modularity and, if the result is negative, returns one of the following certificates: An element of $D(A) \setminus \{a,b,0\}$, $\gcd(D(A)) \neq \gcd(\{a,b\})$, or a set $D' \subsetneq \{a,b,0\}$ such that $D(A) = D'$.
\end{lemma}
\begin{proof}
	Applying Lemmata~\ref{lemma:efficient_non_degeneracy_recognition} and \ref{lemma:wlog_gcd_1} allows us to assume w.l.o.g. that $0 \in D(A)$ and that $\gcd(\{a,b\}) = 1$. Note that the case $a=b$ implies $a=b=1$. Then, testing for $\{1,0\}$-modularity can be accomplished by first transforming $A$ to~\eqref{eq:HNF} and by subsequently testing whether the transformed matrix is \tu.
	
	Thus, assume that $a > b$. Since $\gcd(\{a,b\}) = 1$, $2 \cdot b \neq a \Leftrightarrow \{a,b\} \neq \{2,1\}$. Therefore, the numbers $a$ and $b$ fulfill the prerequisites of Theorem~\ref{thm:format_one_sum}. Follow the proof of Theorem~\ref{thm:format_one_sum} to transform $A$ to~\eqref{eq:full_one_sum_layout}. If the matrix is $\{a,b,0\}$-modular, we will arrive at a representation of the form~\eqref{eq:full_one_sum_layout}. Otherwise, the proof of Theorem~\ref{thm:format_one_sum} (as it is constructive) exhibits a certificate of the following form: An element of $D(A) \setminus \{a,b,0\}$, $\gcd(D(A)) \neq \gcd(\{a,b\})$, or a set $D' \subsetneq \{a,b,0\}$ such that $D(A) = D'$. Without loss of generality, we may further assume that at least one $n \times n$ subdeterminant of $A$ is equal to $\pm a$ and that at least one $n \times n$ subdeterminant of $A$ is equal to $\pm b$, i.e., that $\{a,b,0\} \subseteq D(A)$.\\
	
	Next, we show that i) holds if and only if both ii) and iii) hold, where
	\begin{enumerate}
		\item every nonsingular $n \times n$ submatrix of $A$ has determinant $\pm a$ or $\pm b$,
		\item every nonsingular $n \times n$ submatrix of $A_{I_a \cup I_0, \cdot}$ has determinant $\pm a$,
		\item every nonsingular $n \times n$ submatrix of $A_{I_b \cup I_0, \cdot}$ has determinant $\pm b$.
	\end{enumerate}
	$A_{I_0,\cdot}$ contains the first $n-1$ unit vectors and hence, $A_{I_a \cup I_0, \cdot}$ and $A_{I_b \cup I_0, \cdot}$ admit full column rank.
	
	We first show that ii) and iii) follow from i). Let us start with iii). For the purpose of contradiction, assume that i) holds but not iii). By construction, the $n$-th column of $A_{I_b \cup I_0,\cdot}$ is divisible by $b$. Denote by $A'$ the matrix which we obtain by dividing the last column of $A_{I_b \cup I_0 ,\cdot}$ by $b$. The entries of $A'$ are all equal to $\pm 1$ or $0$. As we have assumed that iii) is invalid, there exists a nonsingular $n \times n$ submatrix of $A'$ whose determinant is not equal to $\pm 1$. In particular, $A'$ is not \tu. By~\cite[Theorem~19.3]{Schrijver_IP}, as the entries of $A'$ are all $\pm 1$ or $0$ but it is not \tu, it contains a submatrix of determinant $\pm 2$. Since $A'_{\cdot,[n-1]}$ is \tu, this submatrix must involve entries of the $n$-th column of $A'$. Thus, it corresponds to a submatrix of $A_{I_b \cup I_0 ,\cdot}$ of determinant $\pm 2 \cdot b$. Append a subset of the first $n-1$ unit vectors to extend this submatrix to an $n \times n$ submatrix. This $n \times n$ submatrix is also an $n \times n$ submatrix of $A$. Its determinant is $\pm 2 \cdot b$, which is not contained in $\{\pm a,\pm b,0\}$ because $2 \cdot b \neq a$, a contradiction to i). For ii), the same argument yields an $n \times n$ subdeterminant of $\pm 2 \cdot a$, which is also a contradiction to i).
	
	Next, we prove that i) holds if both ii) and iii) hold. This follows from the $1$-sum structure of $A$. Let $B$ be any nonsingular $n \times n$ submatrix of $A$. $B$ is of the form
	\[ B = \begin{bmatrix} C & \ & \sfrac{0}{a} \\
	\ & D & \sfrac{0}{b} \end{bmatrix}, \]
	where $C \in \Z^{m_1 \times n_1}$, $D \in \Z^{m_2 \times n_2}$ for $n_1$, $n_2$, $m_1$, and $m_2$ satisfying $n_1 + n_2 + 1 = n = m_1 + m_2$. As $B$ is nonsingular, $n_i \leq m_i$ and $m_i \leq n_i + 1$, $i \in \{1,2\}$. Thus, $n_i \leq m_i \leq n_i + 1$, $i \in \{1,2\}$, and we identify two possible cases which are symmetric: $m_1 = n_1 + 1$, $m_2 = n_2$ and $m_1 = n_1$, $m_2 = n_2 + 1$. We start with the analysis of the former. By Laplace expansion applied to the last column,
	\[ \det \begin{bmatrix} C & \ & \sfrac{0}{a} \\
	\ & D & \sfrac{0}{b} \end{bmatrix} = \det \begin{bmatrix} C & \ & \sfrac{0}{a} \\
	\ & D & 0 \end{bmatrix} \pm \det \begin{bmatrix} C & \ & 0 \\
	\ & D & \sfrac{0}{b} \end{bmatrix}.\]
	The latter determinant is zero as $m_1 = n_1 + 1$. As the former matrix is block-diagonal, $|\det B| = |\det [C \mid \sfrac{0}{a}]| \cdot | \det D|$. $B$ is nonsingular and $D$ is \tu, therefore $|\det D| = 1$. $[C \mid \sfrac{0}{a}]$ is a submatrix of $A_{I_a \cup I_0,\cdot}$ which can be extended to an $n \times n$ submatrix of the same determinant in absolute value by appending a subset of the $n-1$ unit vectors contained in $A_{I_a \cup I_0,\cdot}$. Therefore by ii), $|\det B| = |\det [C \mid \sfrac{0}{a}]| = a$. In the case $m_1 = n_1$, $m_2 = n_2 + 1$ a symmetric analysis leads to $|\det B| = |\det [ D \mid \sfrac{0}{b}]| = b$ by iii).\\
	
	To test for ii), let $A'$ be the matrix which we obtain by dividing the $n$-th column of $A_{I_a \cup I_0, \cdot}$ by $a$. Then, ii) holds if and only if every nonsingular $n \times n$ submatrix of $A'$ has determinant $\pm 1$. Transform $A'$ to~\eqref{eq:HNF}. The topmost $n$ rows of $A'$ must form a unit matrix. Therefore, ii) is equivalent to $A'$ being \tu, which can be tested efficiently. Testing for iii) can be done analogously.\qed
\end{proof}
We now have all the necessary ingredients to prove Theorem~\ref{thm:abc_recognition}. In essence, what remains is a technique to find sample values $a$, $b$, and $c$ for which we test whether $A$ is $\{a,b,c\}$-modular using our previously established algorithms.
\begin{proof}[Proof of Theorem~\ref{thm:abc_recognition}]
	Apply Lemma~\ref{lemma:efficient_non_degeneracy_recognition} for $d = 3$. Either, the algorithm calculates $D(A)$, or gives a certificate that $|D(A)| \geq 4$, or that $0 \in D(A)$. In the former two cases we are done. Assume therefore that $0 \in D(A)$. By assumption, $A$ has full column rank. Thus, find $k_1 \in D(A)$, $k_1 \neq 0$, using Gaussian Elimination. Apply Lemma~\ref{lemma:efficient_0ab_modularity_recognition} to check whether $D(A) = \{k_1,0\}$. Otherwise, $|D(A)| \geq 3$. At this point, a short technical argument is needed to determine an element $k_2 \in D(A) \setminus \{k_1,0\}$.\\
	
	To do so, we first apply a technique from the proof of Lemma~\ref{lemma:wlog_gcd_1}. This technique allows us to reduce finding $k_2 \in D(A) \setminus \{k_1,0\}$ to finding $k_2 \in D(A') \setminus \{\frac{k_1}{\gcd(D(A))},0\}$ for a matrix $A' \in \Z^{m \times n}$ which satisfies $\gcd(D(A')) = 1$: Calculate the Smith normal form of $A$, i.e., calculate $P \in \Z^{m \times m}$ and $Q \in \Z^{n \times n}$ unimodular such that $PAQ = \begin{bmatrix} S \\ 0 \end{bmatrix}$, where $S$ is a diagonal matrix and $\prod_{i=1}^n S_{i,i} = \gcd(D(A))$. This can be done in time polynomial in $n$, $m$, and $\log ||A||_{\infty}$, see~\cite{Storjohann_PhD}. Then, the matrix $A' := AQS\inv$ is integral, satisfies $D(A') = \frac{1}{\gcd(D(A))} \cdot D(A)$ and consequently has the desired property.
	
	Transform $A'$ to~\eqref{eq:HNF}. Then, $A'_{\cdot,n} \geq 0$. If $A'_{\cdot,n}$ has two nonzero entries, say $A'_{p,n}$ and $A'_{q,n}$, such that $A'_{p,n} \neq A'_{q,n}$, then $\det(A'_{[n-1] \cup p,\cdot}) \neq \det(A'_{[n-1] \cup q,\cdot})$, both of which are nonzero. Thus, either $\det(A'_{[n-1] \cup p,\cdot}) \neq \frac{k_1}{\gcd(D(A))}$ or $\det(A'_{[n-1] \cup q,\cdot}) \neq \frac{k_1}{\gcd(D(A))}$. If such entries do not exist, then $A'_{\cdot,n} \in \{0,1\}^m$ as otherwise, $A'_{\cdot,n}$ would be divisible by an integer larger than one, contradicting $\gcd(D(A')) = 1$. This implies that $A'_{[n],\cdot} = \mathcal I_n$, and in particular that every nonsingular submatrix of $A'$ can be extended to an $n \times n$ submatrix of the same determinant in absolute value. Test $A'$ for total unimodularity. Since we know that $|D(A')| = |D(A)| \geq 3$, $A'$ is not \tu, i.e., the algorithm returns a submatrix of determinant at least $2$ in absolute value. The absolute value of this subdeterminant and $\det(\mathcal I_n) = 1$ are two nonzero elements of $D(A')$, one of which cannot be equal to $\frac{k_1}{\gcd(D(A))}$.\\
	
	Assume w.l.o.g. that $k_1 > k_2$. If $2 \cdot k_2 = k_1$, then $A$ has duplicative relations. If not, test $A$ for $\{k_1,k_2,0\}$-modularity using Lemma~\ref{lemma:efficient_0ab_modularity_recognition}. As $\{k_1, k_2,0\} \subseteq D(A)$, either, this algorithm returns that $D(A) = \{k_1,k_2,0\}$, or a certificate of the form $\gcd(D(A)) \neq \gcd(\{k_1,k_2\})$, or an element of $D(A) \setminus \{k_1,k_2,0\}$. In the first case we are done and in the second and third case it follows that $|D(A)| \geq 4$.\qed
\end{proof}
	
	\section{Proof of Theorem \ref{thm:standard_optimize}}
	\label{sct:optimization}

One ingredient to the proof of Theorem~\ref{thm:standard_optimize} is the following result by
Gribanov, Malyshev, and Pardalos~\cite{gribanov2020parametric} which reduces the standard form \ip~\eqref{eq:standard_form_ip_statement} to an \ip in inequality form in dimension $n-m$ such that the subdeterminants of the constraint matrices are in relation.
\begin{lemma}[{\cite[Corollary~1.1, Remark~5, Theorem~3]{gribanov2020parametric}\footnote{Note that in~\cite{gribanov2020parametric}, the one-to-one correspondence between $D(C)$ and $D(B\T)$ is not explicitly stated in Corollary~1.1 but follows from Theorem~3.}}] \label{lemma:corollary_1_gribanov}
	In time polynomial in $n$, $m$, and $\log ||B||_{\infty}$, \eqref{eq:standard_form_ip_statement} can be reduced to the inequality form \ip
	\begin{align} \label{eq:inequality_form_ip_gcd_1}
		\max \{h\T y \colon Cy \leq g,\ y \in \Z^{n-m} \},
	\end{align}
	where $h \in \Z^{n-m}$, $g \in \Z^n$, and $C \in \Z^{n \times (n-m)}$, with $D(C) = \frac{1}{\gcd(D(B\T))} \cdot D(B\T)$.
\end{lemma}
To prove this reduction, the authors apply a theorem by Shevchenko and Veselov~\cite{veselovshevchenkodetidentity} which was originally published in Russian. For completeness of presentation, we will provide an alternative but similar proof of Lemma~\ref{lemma:corollary_1_gribanov} which uses the following well-known determinant identity instead of the aforementioned result.
\begin{lemma}[{Jacobi's complementary minor formula, see~\cite[Lemma A.1e]{caracciolo2013algebraic}}] \label{lemma:jacobis_identity}
	Let $A \in \Z^{n \times n}$ be invertible and $I, J \subseteq [n]$, $|I| = |J| = k$ for $k \in [n]$. Then,
	\[ \det(A_{I,J}) = \det (A) \cdot (-1)^{\sum_{i \in I} i + \sum_{j \in J} j} \cdot \det\left(A \inv_{\overline J,\overline I}\right),\]
	where $\overline I := [n] \setminus I$ and $\overline J := [n] \setminus J$.
\end{lemma}
\begin{proof}[Alternative proof of Lemma~\ref{lemma:corollary_1_gribanov}]
	We closely follow the proof of~\cite{gribanov2020parametric}. First, we reformulate~\eqref{eq:standard_form_ip_statement} in such a way that the $\gcd$ of the full rank subdeterminants of the constraint matrix becomes $1$. To this end, calculate the Smith normal form of $B$, i.e., find $P \in \Z^{m \times m}$ and $Q \in \Z^{n \times n}$ unimodular and nonsingular such that $B = P[S \mid 0]Q$, where $S \in \Z^{m \times m}$ is a diagonal matrix satisfying $\prod_{i=1}^m S_{i,i} = \gcd(D(B\T))$. This can be done in time polynomial in $m$, $n$, and $\log ||B||_{\infty}$, see~\cite{Storjohann_PhD}. Thus, $Bx = b \Leftrightarrow [\mathcal I_m \mid 0] Q x = S \inv P \inv b$. For simplicity, set $b' := S \inv P \inv b$. If $b' \notin \Z^m$, then \eqref{eq:standard_form_ip_statement} is infeasible. Thus, assume that $b' \in \Z^m$. Summarizing, solving~\eqref{eq:standard_form_ip_statement} is equivalent to solving
	\begin{align}
		\label{eq:standard_form_ip_gcd_1}
		\max \{ c\T x \colon [\mathcal I_m \mid 0] Q x = b',\ x \in \Z^n_{\geq 0}\},
	\end{align}
	where due to multiplying by $S \inv$, $D(([\mathcal I_m \mid 0] Q)\T) = \frac{1}{\gcd(D(B\T))} \cdot D(B\T)$.\\
	
	Secondly, we reduce~\eqref{eq:standard_form_ip_gcd_1} to an \ip in inequality form. Since $Q$ is unimodular, substituting $z := Qx$ yields
	\begin{align*}
		& \ \{x \in \Z^n \colon [\mathcal I_m \mid 0] Q x = b'\}\\ = & \ Q \inv \{ z \in \Z^n \colon z_{[m]} = b'\}\\ = & \ \{Q\inv_{\cdot,[m]} b' + Q\inv_{\cdot,[n] \setminus [m]} y \colon y \in \Z^{n-m} \}.
	\end{align*}
	Plugging this identity into~\eqref{eq:standard_form_ip_gcd_1} yields
	\begin{align*}
		& \quad \max \{c \T (Q\inv_{\cdot,[m]} b' + Q\inv_{\cdot,[n] \setminus [m]} y) \colon Q\inv_{\cdot,[m]} b' + Q\inv_{\cdot,[n] \setminus [m]} y \geq 0,\ y \in \Z^{n-m}\}\\
		= & \quad c \T Q\inv_{\cdot,[m]} b' + \max \{ h\T y \colon C y \leq g,\ y \in \Z^{n-m}\},
	\end{align*}
	where $h\T := c\T Q\inv_{\cdot,[n] \setminus [m]}$, $g := Q\inv_{\cdot,[m]} b'$, and $C := -Q\inv_{\cdot,[n] \setminus [m]}$.
	
	Recall that $D(([\mathcal I_m \mid 0] Q)\T) = \frac{1}{\gcd(D(B\T))} \cdot D(B\T)$. As $[\mathcal I_m \mid 0] Q = Q_{[m],\cdot}$, it remains to show that $D((Q_{[m],\cdot})\T) = D(Q\inv_{\cdot,[n] \setminus [m]})$. Lemma~\ref{lemma:jacobis_identity} applied to $A := Q$ for $I := [m]$ and $J\subseteq [n]$, $|J|=m$, yields
	\[ |\det(Q_{[m],J})| = |\det(Q \inv _{\overline J,[n] \setminus [m]})|, \]
	i.e., the claim follows.\qed
\end{proof}

As a second ingredient to the proof of Theorem~\ref{thm:standard_optimize}, we will make use of some results for \emph{bimodular integer programs} (\bips). \bips are \ips of the form
\[ \max \{c \T x \colon Ax \leq b,\ x \in \Z^n\},\]
where $c \in \Z^n$, $b \in \Z^m$, and $A \in \Z^{m \times n}$ is \emph{bimodular}, i.e., $\rank(A) = n$ and $D(A) \subseteq \{2,1,0\}$. As mentioned earlier,~\cite{artmann2017strongly} proved that \bips can be solved in strongly polynomial time. Their algorithm uses the following structural result for \bips by~\cite{veselov2009Intprobimmat} which will also be useful to us.
\begin{theorem}[{\cite[Theorem~2]{veselov2009Intprobimmat}, as formulated in~\cite[Theorem~2.1]{artmann2017strongly}}] \label{thm:veselov_structural}
	Assume that the linear relaxation $\max \{c\T x \colon Ax \leq b,\ x \in \R^n\}$ of a \bip is feasible, bounded and has a unique optimal vertex solution $v$. Denote by $I \subseteq [m]$ the indices of the constraints which are tight at $v$, i.e., $A_{I,\cdot} v = b_I$. Then, an optimal solution $x^*$ of $\max \{c \T x \colon A_{I,\cdot} x \leq b_I,\ x \in \Z^n\}$ is also optimal for the \bip.
\end{theorem}
\begin{proof}[Proof of Theorem~\ref{thm:standard_optimize}]
	Using Lemma~\ref{lemma:corollary_1_gribanov}, we reduce the standard form \ip~\eqref{eq:standard_form_ip_statement} to~\eqref{eq:inequality_form_ip_gcd_1}. Note that $\gcd(D(C)) = 1$. Let us denote~\eqref{eq:inequality_form_ip_gcd_1} with objective vector $h \in \Z^{n-m}$ by $\ineq h$ and its natural linear relaxation by $\ineqr h$. We apply Theorem~\ref{thm:abc_recognition} to $C$ and perform a case-by-case analysis depending on the output.
	\begin{enumerate}
		\item The algorithm calculates and returns $D(C)$. If $0 \notin D(C)$, $C$ is nondegenerate and $\ineq{h}$ can be solved using the algorithm in~\cite{note2016}. Thus, assume that $0 \in D(C)$. $C$ has no duplicative relations. As $\gcd(D(C)) = 1$, this implies that $C$ is $\{a,b,0\}$-modular for $a \geq b > 0$, where $\gcd(\{a,b\}) = 1$ and $(a,b) \neq (2,1)$. Thus, $C$ satisfies the assumptions of Theorem~\ref{thm:format_one_sum}. As a consequence of Theorem~\ref{thm:format_one_sum}, there exist elementary column operations which transform $C$ such that its first $n-m-1$ columns are \tu, i.e., there is $U \in \Z^{(n-m) \times (n-m)}$ unimodular such that $CU = [T \mid d]$, where $T$ is \tu and $d \in \Z^n$. Substituting \mbox{$z := U \inv y$} yields the equivalent problem
		\begin{align} \label{eq:inequality_form_ip}
			\max \{h \T U z \colon [T \mid d] z \leq g,\ z \in \Z^{n-m}\},
		\end{align}
		where we have used that $y = Uz \in \Z^{n-m} \Leftrightarrow z \in \Z^{n-m}$ as $U$ preserves integrality. Let $z^\ast$ be an optimal solution to the mixed-integer linear program
		\begin{align} \label{eq:milp}
			\max \{h \T U z \colon [T \mid d] z \leq g,\ z \in \R^{n-m},\ z_{n-m} \in \Z\},
		\end{align}
		which can be found in polynomial time~\cite[Chapter~18.4]{Schrijver_IP}. If no such solution exists,~\eqref{eq:inequality_form_ip} is infeasible. Fixing $z_{n-m} := z^\ast_{n-m}$ in~\eqref{eq:milp} induces an \lp in dimension $n-m-1$.
		Let $\bar z$ be a vertex solution to this \lp, which can be found efficiently (see, for example, \cite{grotschel2012geometric}).\footnote{As $T$ has full column rank, the feasible region is pointed, i.e., such a vertex exists.} Since $T$ is \tu, $\bar z \in \Z^{n-m-1}$. The solution $[\bar z \mid z^\ast_n]$ has the same objective value as $z^*$ and is optimal for~\eqref{eq:inequality_form_ip} since it is integral and~\eqref{eq:milp} is a relaxation of~\eqref{eq:inequality_form_ip}.\\
		
		\item The algorithm returns that $|D(C)| \geq 4$. Then, $|D(B\T)| = |D(C)| \geq 4$.\\
		
		\item The algorithm returns a duplicative relation, i.e., $\{2 \cdot k, k\} \subseteq D(C)$, $k > 0$. This case is more involved because we do not have any information as to which other elements might be contained in $D(C)$.
		
		Assume w.l.o.g. that $\ineqr h$ is feasible and that $\ineqr h$ is bounded. We postpone the unbounded case to the end of the proof. Calculate an optimal vertex solution $v$ to $\ineqr{h}$. If $v \in \Z^{n-m}$, then $v$ is also optimal for $\ineq h$. Thus, assume that $v \not\in \Z^{n-m}$ and let $I \subseteq [n]$ be the indices of tight constraints at $v$, i.e., $C_{I,\cdot} v = g_I$. In what follows, we prove that we may assume w.l.o.g. that (a) $0 \in D(C)$, (b) $k=1$, and (c) every nonzero $(n-m) \times (n-m)$ subdeterminant of $C_{I,\cdot}$ is equal to $\pm 2$.\footnote{In particular, (c) implies that $C_{I,\cdot}$ is bimodular.}\\
		
		\begin{enumerate}
			\item From Lemma~\ref{lemma:efficient_non_degeneracy_recognition} applied to $C$ for $d = 3$ we obtain three possible results: $|D(C)| \geq 4$, $0 \notin D(C)$ or $0 \in D(C)$. In the first case we are done and in the second case, $C$ is nondegenerate and $\ineq{h}$ can be solved using the algorithm in~\cite{note2016}. Therefore, w.l.o.g., $0 \in D(C)$.\\
			\item If $\{2 \cdot k,k,0\} \subseteq D(C)$ for $k > 1$, it follows from $\gcd(D(C)) = 1$ that $|D(C)|\geq 4$. Therefore, w.l.o.g., $k=1$.\\
			\item Since $v \notin \Z^{n-m}$, it holds that $1 \notin D(C_{I,\cdot})$ as otherwise, $v \in \Z^{n-m}$ due to Cramer's rule. Apply Theorem~\ref{thm:abc_recognition} once more, but this time to $C_{I,\cdot}$. If the algorithm returns that $|D(C_{I,\cdot})| \geq 4$, then $|D(C)| \geq 4$. If the algorithm returns a duplicative relation, i.e., $\{2 \cdot s,s\} \subseteq D(C_{I,\cdot})$, then $s \neq 1$ as $1 \notin D(C_{I,\cdot})$. Since by (a) and (b), $\{2,1,0\} \subseteq D(C)$, it follows that $\{2\cdot s, 2, 1, 0\} \subseteq D(C)$. Thus, $|D(C)| \geq 4$. If the algorithm calculates and returns $D(C_{I,\cdot})$, then it either finds that every nonzero $(n-m) \times (n-m)$ subdeterminant of $C_{I,\cdot}$ is equal to $\pm 2$ or it finds an element $t \in D(C_{I,\cdot}) \setminus \{2,0\}$. In the latter case, $t \neq 1$ as $1 \notin D(C_{I,\cdot})$, implying that $\{t,2,1,0\} \subseteq D(C)$ and $|D(C)| \geq 4$.\\
		\end{enumerate}
	
		Let $\ineqcone{h} := \max \{h \T y \colon C_{I,\cdot} y \leq g_I,\ y \in \Z^{n-m}\}$. As $C_{I,\cdot}$ is bimodular, this is a \bip. By possibly perturbing the vector $h$ (e.g. by adding $\frac{1}{M} \cdot \sum_{i \in I} C_{i,\cdot}$ for a sufficiently large $M > 0$), we can assume that $v$ is the unique optimal solution to $\ineqr h$, which will allow us to apply Theorem~\ref{thm:veselov_structural}. Solve $\ineqcone{h}$ using the algorithm by~\cite{artmann2017strongly}. If $\ineqcone{h}$ is infeasible, so is $\ineq{h}$. Let $y \in \Z^{n-m}$ be an optimal solution for $\ineqcone{h}$. It follows that either, $y$ is also optimal for $\ineq{h}$ or that $|D(C)| \geq 4$: If $y$ is feasible for $\ineq{h}$, it is also optimal since $\ineqcone{h}$ is a relaxation of $\ineq{h}$. If $C$ is bimodular, Theorem~\ref{thm:veselov_structural} states that $y$ is feasible for $\ineq{h}$, i.e., it is optimal for $\ineq{h}$. Thus, if $y$ is not feasible for $\ineq{h}$, $D(C)$ contains an element which is neither $0$ nor $1$ nor $2$. As $\{2,1,0\} \subseteq D(C)$ by (a) and (b), this implies that $|D(C)| \geq 4$.\\
		
		It remains to explain why we may assume that $\ineqr h$ is bounded. If not, $\ineq h$ is either infeasible or unbounded. More precisely, $\ineq h$ is unbounded if and only if $\{y \in \Z^{n-m} \colon Cy \leq g\}$ is feasible. We reduce the feasibility test of this set to a bounded \ip of the same form as above: Set $s := C_{1,\cdot}$. By construction, $\ineqr s$ is bounded. Solve $\ineq s$ using our algorithm above. Either, we determine a feasible point of $\{y \in \Z^{n-m} \colon Cy \leq g\}$ in which case $\ineq h$ is unbounded, we find that this set is infeasible, or we find that \mbox{$|D(C)| \geq 4$}.
	\end{enumerate}\qed
\end{proof}
	
	\bibliographystyle{spmpsci}
	\bibliography{bibliography}
	
\end{document}